\documentclass[reqno]{amsart}
\usepackage{hyperref}

\AtBeginDocument{{\noindent\small
\emph{Electronic Journal of Differential Equations},
Vol. 2020 (2020), No. 111, pp. 1--14.\newline
ISSN: 1072-6691. URL: http://ejde.math.txstate.edu or http://ejde.math.unt.edu}
\thanks{\copyright 2020 Texas State University.}
\vspace{8mm}}

\begin{document}
\title[\hfilneg EJDE-2020/111\hfil Fractional differential equations]
{Convergence of solutions of fractional differential equations
to power-type functions}

\author[M. D. Kassim, N. E. Tatar \hfil EJDE-2020/111\hfilneg]
{Mohammed Dahan Kassim, Nasser Eddine Tatar}

\address{Mohammed Dahan Kassim \newline
Department of Basic Engineering Sciences,
College of Engineering,
Imam Abdulrahman Bin Faisal University,
P.O. Box 1982, Dammam 31441, Saudi Arabia}
\email{mdkassim@iau.edu.sa}

\address{Nasser Eddine Tatar\newline
King Fahd University of Petroleum and Minerals,
Department of Mathematics and Statistics,
Dhahran, 31261, Saudi Arabia}
\email{tatarn@kfupm.edu.sa}

\thanks{Submitted September 15, 2020. Published November 4, 2020.}
\subjclass[2010]{34E10, 26A33, 34A08}
\keywords{Asymptotic behavior; boundedness; fractional differential equation;
\hfill\break\indent Caputo fractional derivative; Riemann-Liouville
fractional derivative}

\begin{abstract}
 In this article we study the asymptotic behavior of solutions of some fractional
 differential equations.  We prove convergence to power type functions under some
 assumptions on the nonlinearities. Our results extend and generalize some existing
 well-known results on solutions of ordinary differential equations.
 Appropriate estimations and lemmas such as a fractional version of L'Hopital's rule
 are used.
\end{abstract}

\maketitle
\numberwithin{equation}{section}
\newtheorem{theorem}{Theorem}[section]
\newtheorem{lemma}[theorem]{Lemma}
\newtheorem{definition}[theorem]{Definition}
\newtheorem{remark}[theorem]{Remark}
\newtheorem{example}[theorem]{Example}
\allowdisplaybreaks

\section{Introduction}

We consider the initial value problems
\begin{equation}
\begin{gathered}
(^C \mathfrak{D}_0^{\alpha }x) '(\tau)
 =f(\tau ,x(\tau) ,^C \mathfrak{D}_0^{\beta}x(\tau)),\quad
 0<\beta <\alpha <1,\; \tau >0 \\
{}^C \mathfrak{D}_0^{\alpha }x(\tau) \big|_{\tau=0} =b_2,\quad
=x(\tau) |_{\tau =0}=b_1,\quad b_1, \quad b_2\in \mathbb{R},
\end{gathered}  \label{cd1}
\end{equation}
and
\begin{equation}
\begin{gathered}
^C \mathfrak{D}_0^{\alpha }x(\tau)
=f(\tau ,x(\tau) ,{}^C \mathfrak{D}_0^{\beta }x(\tau)) ,
\quad 0\leq \beta <\alpha <1,\; \tau >0 \\
x(\tau) |_{\tau =0}=b,
\end{gathered}   \label{sc1}
\end{equation}
where $^C \mathfrak{D}_0^{\alpha }$ is the Caputo fractional
derivative. The definition of the Caputo fractional derivative is given in
the next section. We prove that the solutions of \eqref{cd1} approach power
type functions and the solutions of \eqref{sc1} are bounded. To this end,
the fractional differential problems \eqref{cd1} and \eqref{sc1} are first
transformed into equivalent integral equations in appropriate underlying
spaces. Various appropriate estimates, comparison theorems and lemmas are
used. Moreover, we prove a Caputo fractional version of L'Hopital's rule.
Our arguments here are quite different from those used so far in the
literature.

The behavior of solutions of various classes of ODEs (ordinary differential
equations) has been discussed in fairly a large number of papers in the
literature. For example the equation
\begin{equation}
x''(\tau) +f(\tau ,x(\tau)) =0,  \label{ra3}
\end{equation}
has been studied in
\cite{Kusano+Trench-1,Kusano+Trench-2,Cohen,Tong,Waltman,Con,Trench} and other
papers. The authors proved that, under various conditions, all solutions of
\eqref{ra3} are asymptotic to $c\tau +b$ as $\tau \to \infty $, $c,b\in\mathbb{R}$.
For the equation
\begin{equation}
x''(\tau) +f(\tau ,x(\tau),x'(\tau)) =0,  \label{ra4}
\end{equation}
see, for instance
\cite{Constantin,Dannan, lip,Med(2016),Med(2008),mus,Rogovchenko,Rogovchenko+Rogovchenko}.
It is proved that all solutions of \eqref{ra4} can be expressed
asymptotically as $c\tau +b$ as $\tau \to \infty $, $c, b\in\mathbb{R}$.

Medve\v{d} and Pek\'arkov\'a \cite{Med(2008)}, studied the
one-dimensional p-Laplacian equation
\begin{equation}
(| x'|^{p-1}x') '=f(\tau ,x,x') ,\quad p>1.  \label{1.6}
\end{equation}
They demonstrated that any solution of \eqref{1.6} behaves asymptotically as
$b+c\tau $ as $\tau \to \infty $ for some real numbers $b$, $c$.

In  \cite{Med(2016)}, the initial value problem
\begin{equation}
\begin{gathered}
(\Phi _{p}(x') \Psi (\tau))
'+f(\tau ,x,x') =0,\quad 1<p<2, \\
x(\tau _0) =x_0,\quad x'(\tau _0)
=x_1,\quad \tau _0\geq 1,
\end{gathered}  \label{1.7}
\end{equation}
\newline
was studied, where $\Phi _{p}(u) =| u|^{p-2}u$
and $\Psi (\tau) $ is a continuous positive function.
Sufficient conditions under which all solutions of \eqref{1.7} obey the
asymptotic expansion $x(\tau) =b+c\tau $ are established.

In contrast, the fractional case of equations \eqref{ra3} and \eqref{ra4}
have been studied by comparatively a only few researchers; see, for instance
\cite{B1,B2,B3,B4,B5,Bre,K2,K3,K4,K5, K1,Medved-12,Medved-13,Med(2015)}.
In 2009, B\u{a}leanu and Mustafa \cite{B1} studied the nonlinear fractional
differential equation
\begin{equation}
^C \mathfrak{D}_0^{\alpha }x(\tau) =f(\tau ,x(
\tau)) ,\quad 0<\alpha <1,\quad \tau >0.  \label{1.1}
\end{equation}
They showed that the solutions of \eqref{1.1} are asymptotic to
$o(\tau^{c\alpha }) $ as $\tau \to \infty $, for some $c$.

In 2012, Medve\v{d} \cite{Medved-12} studied the problem
\begin{equation}
\begin{gathered}
^C \mathfrak{D}_{a}^{\alpha +1}x(\tau)
=f(\tau,x(\tau)) ,\quad 0<\alpha <1,\quad \tau \geq a>1 \\
x(a) =c_1,\quad x'(a) =c_2.
\end{gathered}  \label{1.2}
\end{equation}
He demonstrated that any solution of \eqref{1.2} has the asymptotic property
$x(\tau) =b+c\tau $ as $\tau \to \infty $, for some $c,b\in\mathbb{R}$.

Also, in 2013, Medve\v{d} \cite{Medved-13} discussed the equation
\begin{equation}
^C \mathfrak{D}_{a}^{\alpha +1}x(\tau)
=f(\tau,x(\tau) ,x'(\tau)),\quad
0<\alpha <1,\quad \tau \geq a>1  \label{1.3}
\end{equation}
and proved that every solution of \eqref{1.3} can be expressed
asymptotically as $b+c\tau $ as $\tau \to \infty $, for some
$c,b\in\mathbb{R}$.

Brestovansk\'a and Medve\v{d} \cite{Bre} studied the problem
\begin{equation}
\begin{gathered}
x''(\tau) +f(\tau ,x(\tau),x'(\tau)) +\sum_{i=1}^{m}r_i(
\tau) \int_0^{\tau }(\tau -s)^{\alpha_i-1}f_i(s,x(s) ,x'(s)) ds=0 \\
x(1) =b_1,\quad x'(1) =b_2,\quad 0<\alpha _i<1,\; i=1,2,\dots ,m.
\end{gathered}  \label{1.4}
\end{equation}
They showed that any solution  enjoys the asymptotic expansion
$x(\tau) =b+c\tau $ as $\tau \to \infty $, for some $c,b\in\mathbb{R}$.

In 2015, Medve\v{d} and Posp\'{\i}\v{s}il \cite{Med(2015)} considered the
equation
\begin{equation}
^C \mathfrak{D}_{a}^{\alpha }x(\tau) =f(\tau ,x(
\tau) ,{}^C \mathfrak{D}_0^{\beta }x(\tau)) ,
\quad \tau >a.  \label{1.5}
\end{equation}
They proved that any solution $x(\tau) $ with  $
0<\beta <\alpha <1$, has the asymptotic property $x(\tau)
=c\tau^{\beta }+o(\tau^{\beta }) $ as $\tau \to \infty$, for some
$c\in \mathbb{R}$.
Also they proved that any solution $x(\tau) $ of \eqref{1.5}, $0<\beta <1<\alpha <2$,
has the asymptotic property $x(\tau)=c\tau +o(\tau) $ as $\tau \to \infty $, for some
$c\in\mathbb{R}$. Moreover, they proved that there exists a constant
$c\in\mathbb{R}$ such that any global solution $x(\tau) $ of the initial value
problem
\begin{equation}
\begin{gathered}
^C \mathfrak{D}_{a}^{\alpha }x(\tau)
=f\big(\tau ,x(\tau) ,x'(\tau) ,\dots ,x^{(n-1)
}(\tau) ,{}^C \mathfrak{D}_0^{\beta _1}x(\tau
) ,\dots ,{}^C \mathfrak{D}_0^{\beta _{m}}x(\tau)
\big) ,\\
  \tau >a .
x^{(n-1) }(a) =c_i,\quad i=0,1,\dots ,n-1,
\end{gathered} \label{1.8}
\end{equation}
has the asymptotic property $x(\tau) =c\tau^{k}+o(\tau ^{k}) $ as $\tau \to \infty $, where
$k\in \max \{n-1,\beta _{m}\} $, $0<\beta _1<\dots .<\beta _{m}<\alpha <n$, and
$n$, $m\in\mathbb{N}$.

In Sections \ref{sec2} and \ref{sec3}, we prepare some material which will
be needed later in our proofs.
Sections \ref{sec4}, \ref{sec5} and \ref{sec6}
are devoted to the main results on the asymptotic behavior results and
boundedness of solutions for non-fractional and fractional source terms,
respectively.

\section{Preliminaries} \label{sec2}

In this section, we introduce some basic definitions, notation, properties
and lemmas to be used in our results. We refer the reader to \
cite{Kilbas,Podl-1999,Samko} for more details.

\begin{definition}[\cite{Kilbas}] \label{def:2.1.4} \rm
We introduce the space
\[
C_{\eta }[ a,b] =\{ h:(a,b]\to \mathbb{R}:
\quad h(\tau) (\tau -a)^{\eta }\in C[ a,b] \} ,\quad 1>\eta \geq 0.
\]
\end{definition}

\begin{definition}[\cite{Kilbas}] \label{def:2.2.1}\rm
The left-sided Riemann-Liouville fractional integral of order $\alpha >0$ is
defined by
\[
\mathfrak{I}_{a}^{\alpha }f(\tau) :=\frac{1}{\Gamma (
\alpha) }\int_{a}^{\tau }\frac{f(s) }{(
\tau -s)^{1-\alpha }}ds,\quad \tau >a,
\]
provided that the right hand side exists.
\end{definition}

\begin{definition}[\cite{Kilbas}] \label{def:2.2.3}\rm
The left-sided Riemann-Liouville fractional derivative of order
$\alpha \geq 0$, $n-1\leq \alpha <n$, $n=-[ -\alpha ] $, is defined by
\begin{align*}
\mathfrak{D}_{a}^{\alpha }f(\tau)
&= D^{n}\mathfrak{I}_{a}^{n-\alpha }f(\tau)
 =(\frac{d}{d\tau })^{n}\mathfrak{I}_{a}^{n-\alpha }f(\tau) \\
&=\frac{1}{\Gamma (n-\alpha)} \big(\frac{d}{d\tau }\big)^{n}
 \int_{a}^{\tau }\frac{f(s)}{(\tau -s)^{\alpha -n+1}}ds,\quad \tau>a.
\end{align*}
In particular, when $\alpha =n$ we have
$\mathfrak{D}_{a}^{\alpha}f=D^{n}f$,
and when $\alpha =0$, $\mathfrak{D}_{a}^{0}f=f$.
\end{definition}

\begin{definition}[\cite{Kilbas}] \label{def:2.2.4}\rm
The left-sided Caputo fractional derivative of order $\alpha \geq 0$, $n-1\leq \alpha <n$,
$n=-[ -\alpha ] $, is defined by
\[
^C\mathfrak{D}_{a}^{\alpha }f(\tau) =\mathfrak{I}
_{a}^{n-\alpha }f^{(n) }(\tau),\;\tau >a.
\]
\end{definition}

The fractional integral and fractional derivative of power functions have
the same effect as the integer-order integral and derivative. Namely,

\begin{lemma}[\cite{Kilbas}]\label{pro:2.2.1}
If $\beta >0$ and $\alpha \geq 0$,  then
\begin{gather*}
\mathfrak{I}_{a}^{\alpha }(\tau -a)^{\beta -1}=\frac{\Gamma
(\beta) }{\Gamma (\beta +\alpha) }(\tau
-a)^{\alpha +\beta +1},\quad \alpha >0,\; \tau >a, \\
^C\mathfrak{D}_{a}^{\alpha }(\tau -a)^{\beta -1}=\frac{\Gamma
(\beta) }{\Gamma (\beta -\alpha) }(\tau
-a)^{\beta -\alpha -1},\quad \alpha \geq 0,\; \tau >a.
\end{gather*}
 If $\beta =1$, then
$(^C\mathfrak{D}_{a}^{\alpha }1) (\tau) =0$, $\tau >a$.
\end{lemma}

The Riemann-Liouville fractional integral (Definition \ref{def:2.2.1})
satisfies the following semigroup property.

\begin{lemma} \label[\cite{Kilbas}] \label{lem:2.2.2}
Let $0\leq \eta <1$, $\alpha >0$ and $\beta >0$.  If $h\in C_{\eta }[ a,b]$,
then
\[
\mathfrak{I}_{a}^{\beta }\mathfrak{I}_{a}^{\alpha }h(\tau) =
\mathfrak{I}_{a}^{\beta +\alpha }h(\tau) ,\quad \tau >a.
\]
\end{lemma}

The following result provides another composition of the fractional
integration operator $\mathfrak{I}_{a}^{\alpha }$ with the fractional
differentiation operator $\mathfrak{D}_{a}^{\alpha }$.

\begin{lemma}[\cite{Kilbas}] \label{lem:2.2.5}
Let $\alpha >0$, $0\leq \eta <1$, $n=-[ -\alpha ] $.
If $h\in C_{\eta }[ a,b] $  and
$\mathfrak{I}_{a}^{n-\alpha }h\in C_{\eta}^{n}[ a,b] $, then
\[
\mathfrak{I}_{a}^{\alpha }\mathfrak{D}_{a}^{\alpha }h(\tau)
=h(\tau) -\sum_{i=1}^{n}\frac{(\mathfrak{D}^{n-i}
\mathfrak{I}_{a}^{n-\alpha }h) (a) }{\Gamma (\alpha
-i+1) }(\tau -a)^{\alpha -i},\quad \tau >a.
\]
\end{lemma}

\begin{lemma}[\cite{Kilbas}] \label{lem:2.4.1}
Let $\alpha >0$, $n=- [ -\alpha ] $. If $h\in C^{n}[ a,b] $ or $
h\in AC^{n}[ a,b]$,  then
\[
\mathfrak{I}_{a}^{\alpha } {}^C \mathfrak{D}_{a}^{\alpha }h(
\tau) =h(\tau) -\sum_{k=0}^{n-1}\frac{h^{(k) }(a) }{k!}(\tau -a)^{k},\quad \tau >a.
\]
\end{lemma}

\begin{lemma} \label{lem:4.2.1}
 Let $0<\beta \leq \alpha <1$. If $h\in AC[ a,b]$, then
\[
^C \mathfrak{D}_0^{\beta }h=\mathfrak{I}_0^{\alpha -\beta }{}^C \mathfrak{D}_0^{\alpha }h.
\]
\end{lemma}

\begin{proof}
From Definition \ref{def:2.2.4} and  Lemma \ref{lem:2.2.2}, we have
\[
^C \mathfrak{D}_0^{\beta }h=\mathfrak{I}_0^{1-\beta }h'=
\mathfrak{I}_0^{\alpha -\beta }\mathfrak{I}_0^{1-\alpha }h'=
\mathfrak{I}_0^{\alpha -\beta }{}^C \mathfrak{D}_0^{\alpha }h.
\]
\end{proof}

\begin{lemma}[\cite{K3}] \label{lem:3.1.2}
 Let $f\in L_1(0,\infty) $. Then
\[
\lim_{\tau \to \infty }\frac{1}{\tau^{\alpha }}\mathfrak{I}
_0^{\alpha +1}f(\tau) =\frac{1}{\Gamma (\alpha
+1) }\int_0^{\infty }f(s) ds=\frac{1}{\Gamma (
\alpha +1) }\mathfrak{I}_0^{1}f(\infty) ,\ \alpha >0.
\]
\end{lemma}

\begin{lemma}\label{lem:4.1.1}
Let $0<\alpha <1$ and $0\leq \eta <1$.
 Assume that $x\in AC[0,\infty)$ and $\mathfrak{I}_0^{1-\alpha }x'\in C_{\eta }^{1}[0,\infty)$.
Then
\begin{equation}
\lim_{\tau \to \infty }\frac{x(\tau) }{\tau^{\alpha }}
=\lim_{\tau \to \infty }\frac{^C \mathfrak{D}_0^{\alpha
}x(\tau) }{\Gamma (1+\alpha) }.  \label{C4.5}
\end{equation}
\end{lemma}

\begin{proof}
Since $x\in AC[0,\infty)\subset C_{\eta }[0,\infty)$  and
$\mathfrak{I}_0^{1-\alpha }x'\in C_{\eta }^{1}[0,\infty)$, we can
use Lemma \ref{lem:2.2.5} to obtain
\begin{equation}
\mathfrak{I}_0^{\alpha }\mathfrak{D}_0^{\alpha }x'(\tau
) =x'(\tau) -\frac{\mathfrak{I}_0^{1-\alpha
}x'(0) }{\Gamma (\alpha) }\tau^{\alpha
-1},\quad \tau >0.  \label{C6}
\end{equation}
Applying $\mathfrak{I}_0^{1}$ to both sides of \eqref{C6}, using Lemma
\ref{pro:2.2.1} (with $\beta =\alpha $) and Lemma \ref{lem:2.2.2}, we obtain
\begin{equation}
x(\tau) =x(0) +\frac{\mathfrak{I}_0^{1-\alpha
}x'(0) }{\Gamma (\alpha +1)}\tau^{\alpha }+\mathfrak{I}
_0^{1+\alpha }\mathfrak{D}_0^{\alpha }x'(\tau) ,
\quad \tau >0.  \label{C6.5}
\end{equation}
Dividing both sides of \eqref{C6.5} by $\tau^{\alpha }$, we obtain
\[
\frac{x(\tau) }{\tau^{\alpha }}=\frac{x(0) }{\tau
^{\alpha }}+\frac{\mathfrak{I}_0^{1-\alpha }x'(0) }{
\Gamma (\alpha +1)}+\frac{1}{\tau^{\alpha }}\mathfrak{I}_0^{1+\alpha }
\mathfrak{D}_0^{\alpha }x'(\tau) ,\quad \tau >0.
\]
Next, we take the limit as $\tau \to \infty $, we arrive at
\begin{equation}
\lim_{\tau \to \infty }\frac{x(\tau) }{\tau^{\alpha }}=
\frac{\mathfrak{I}_0^{1-\alpha }x'(0) }{\Gamma (
1+\alpha) }+\frac{1}{\Gamma (1+\alpha) }\lim_{\tau
\to \infty }\mathfrak{I}_0^{1}\mathfrak{D}_0^{\alpha }x'(\tau) ,  \label{C6.66}
\end{equation}
where we  used Lemma \ref{lem:3.1.2}. Moreover, we conclude that
\begin{equation}
\mathfrak{I}_0^{1}\mathfrak{D}_0^{\alpha }x'(\tau)
=\mathfrak{I}_0^{1}D\mathfrak{I}_0^{1-\alpha }x'(\tau
) =\mathfrak{I}_0^{1-\alpha }x'(\tau) -
\mathfrak{I}_0^{1-\alpha }x'(0) ={}^C
\mathfrak{D}_0^{\alpha }x(\tau) -{}^C \mathfrak{D}
_0^{\alpha }x(0) ,\quad \tau >0,  \label{C6.67}
\end{equation}
and \eqref{C4.5} follows directly from \eqref{C6.66} and \eqref{C6.67}.
\end{proof}

\section{Some useful inequalities} \label{sec3}

First, we define the following special classes of functions
\begin{gather}
H_k=\{ h\in L_1(0,\infty) : h\text{ is positive and }
s^{k}h\in L_1(1,\infty) ,\; k>-1\}, \label{space_l} \\
\begin{aligned}
M=\big\{&F:(0,\infty) \times \mathbb{R}
_{+}\to \mathbb{R}_{+}\text{ where }0\leq F(\tau ,s)
-F(\tau ,r) \leq H(\tau) (s-r) ,\\
&\text{for some continuous function $H$ on }\mathbb{R}_{+},\;
s\geq r\geq 0\text{ and }\tau >0\big\}.
\end{aligned}  \label{space_h} \\
\begin{aligned}
\Phi=\big\{&\varphi \in C(0,\infty) :\varphi \text{ is
nondecreasing and positive on }(0,\infty) , \\
&\frac{1}{v}\varphi (w) \leq \varphi
(\frac{w}{v}) ,\; w>0,\; v\geq 1\big\}.
\end{aligned}\label{space_f}
\end{gather}
The proofs of the following lemmas are based on an application of the Bihari
inequality which is a generalization of the Gronwall inequality.

\begin{lemma}[\cite{K3}] \label{lem:3.1.1}
Let $g(\tau) $ and $ z(\tau)$ be nonnegative continuous functions defined for $\tau \geq 0$,
$\varphi \in \Phi$ and $c_i\in \mathbb{R}$, $i=1,2,3$. Then
\begin{equation}
z(\tau) \leq c_1+c_2\tau^{\gamma }+c_3\tau^{\gamma
}\int_0^{\tau }g(s)\varphi (z(s)) ds,\text{\ \ }
\tau ,\quad \gamma \geq 0,  \label{r2.55}
\end{equation}
implies
\begin{equation}
z(\tau) \leq \begin{cases}
E^{-1}(E(| c_1| +|
c_2|) +| c_3| \int_0^{\tau
}g(s)ds) ,& 0\leq \tau <1 \\[4pt]
\tau^{\gamma }E^{-1}(E(A) +| c_3|
\int_1^{\tau }s^{\gamma }g(s)ds) , & \tau \geq 1,
\end{cases}  \label{r2.66}
\end{equation}
where
\begin{gather*}
A=| c_1| +| c_2| +|c_3| \varphi (E^{-1}(C)) \int_0^{1}g(s)ds, \\
C=E(| c_1| +| c_2|) +| c_3| \int_0^{1}g(s)ds<\infty ,
\end{gather*}
and $E^{-1}$ is the inverse function of
\[
E(\mathfrak{\xi }) =\int_{\mathfrak{\xi }_0}^{\mathfrak{\xi }}
\frac{ds}{\varphi (s) }.
\]
\end{lemma}

\begin{lemma}[\cite{K3}] \label{lem:4.2.2}
Let $z(\tau)$ satisfy
\begin{equation}
z(\tau) \leq c_1\tau^{\gamma }+c_2\tau^{\gamma
}\int_0^{\tau }[ F_1(s,z(s) +c_3)
+F_2(s,z(s) +c_{4}) +h(s) ] ds, \quad
\tau \geq 0,  \label{cd25}
\end{equation}
where $h:C[\mathbb{R}_{+},\mathbb{R}_{+}] $,
$F_{j}\in M$, $j=1,2$ and $\gamma $, $c_i>0$, $i=1,2,3,4$. Then
\begin{equation}
z(\tau) \leq \tau^{\gamma }f(\tau) ,\quad \tau >0,  \label{cd26}
\end{equation}
where
\begin{equation}
\begin{aligned}
f(\tau) &=\Big(c_1+c_2\int_0^{\tau }[ F_1(
s,c_3) +F_2(s,c_{4}) +h(s) ]ds\Big) \\
&\quad \times \exp \Big(c_2\int_0^{\tau }s^{\gamma }[ N_1(
s) +N_2(s) ] ds\Big),\quad \tau >0,
\end{aligned} \label{cd27}
\end{equation}
with $N_1$ and $N_2$ are as in the definition of $M$
corresponding to $F_1$ and $F_2$, respectively.
\end{lemma}

\begin{remark} \label{rem:equiv} \rm
If $p$, $q>1$ and $\frac{1}{p}+\frac{1}{q}=1$, then for $
\alpha >0$, $p(\alpha -1)+1>0\Longleftrightarrow q\alpha >1$.
\end{remark}

\begin{lemma}[\cite{K4}] \label{lem:rec2}
If $\upsilon, \lambda +1>1/r$,  for some $r>1$, and $g$ is a
nonnegative continuous function defined on $[0,\infty)$,  then
\begin{equation}
\int_0^{\tau }(\tau -s)^{\upsilon -1}s^{\lambda }g(
s) ds\leq C\tau^{\upsilon +\lambda -1/r}\Big(\int_0^{\tau
}g^{r}(s) ds\Big)^{1/r},\quad \tau >0,  \label{sr66}
\end{equation}
where
\[
C=K_{p(\upsilon -1) +1,p\lambda }=\frac{\Gamma (p\lambda
+1) \Gamma (p(\upsilon -1) +1) }{\Gamma (
p\lambda +p(\upsilon -1) +2) },\quad \frac{1}{p}+\frac{1}{r}=1.
\]
\end{lemma}

\begin{lemma}[\cite{K4}] \label{lem:res1}
Let $h$ and $z$  be nonnegative continuous functions defined on
$[0,\infty)$. Let $\varphi _i(z)>0$ on $(0,\infty)$, $i=1,2$, and
$\varphi _i(z)$ are continuous nondecreasing functions defined on
$[0,\infty)$. If
\begin{equation}
z(\tau) \leq K_1+K_2(\int_0^{\tau }h^{q}(s) \varphi _1^{q}(z(s)) \varphi
_2^{q}(z(s)) ds)^{1/q},\quad q>1, \; \tau >0,  \label{sr14}
\end{equation}
where $K_i\in \mathbb{R}_{+}$, $i=1,2$, then
\[
z(\tau) \leq \Big[ E^{-1}\Big(E(2^{q-1}K_1)
+2^{q-1}K_2\int_0^{\tau }h^{q}(s) ds\Big) \Big]^{1/q}, \quad
\tau >0,
\]
where $E^{-1}$ is the inverse of
\[
E(\mathfrak{\xi }) =\int_{\mathfrak{\xi }_0}^{\mathfrak{\xi }}
\frac{ds}{\varphi _1^{q}(s^{1/q}) \varphi _2^{q}(
s^{1/q}) },\quad \mathfrak{\xi }>\mathfrak{\xi }_0>0.
\]
\end{lemma}

\section{Source without fractional derivatives} \label{sec4}

We consider the asymptotic behavior of solutions of the equation
\begin{equation}
(^C \mathfrak{D}_0^{\alpha }x) '(\tau
) =f(\tau ,x(\tau)),\quad 0<\alpha <1\text{, }\tau
\geq 0,  \label{C1}
\end{equation}
subject to
\begin{equation}
x(0)=b_1 ,\quad
{}^C \mathfrak{D}_0^{\alpha }x(\tau) |_{\tau
=0}=b_2,  \label{C1.5}
\end{equation}
in the space
\begin{equation}
C_{1-\alpha }^{\alpha ,1}[0,\infty)=\{ x\in AC[0,\infty),(^C
\mathfrak{D}_0^{\alpha }x) '\in C_{1-\alpha }[0,\infty)\} .  \label{space of Caputo}
\end{equation}
We assume the following conditions:
\begin{itemize}
\item[(C1)] $f(\tau ,x)\in C[ [0,\infty)\times \mathbb{R},\mathbb{R}] $ is such that
$f(\cdot,x(\cdot))\in C_{1-\alpha }[0,\infty)$ for any $x\in AC[0,\infty)$.

\item[(C2)] There are continuous functions $P,\varphi :[0,\infty)\to \lbrack 0,\infty)$ such that
\begin{equation}
| f(\tau ,x(\tau))| \leq \varphi (
| x(\tau)|) P(\tau),\quad \tau \geq 0,
\label{C2}
\end{equation}
where
\begin{equation}
\int_1^{\infty }s^{\alpha }P(s) ds<\infty ,  \label{C4}
\end{equation}
and $\varphi \in \Phi$.
\end{itemize}

\begin{theorem}\label{thm:4.1.1}
 Suppose $f$ satisfies {\rm (C1), (C2)}  and $x\in AC[0,\infty)$ is a solution
 of \eqref{C1}-\eqref{C1.5}. Then
\[
\lim_{\tau \to \infty }\frac{x(\tau) }{\tau^{\alpha }}
=a\in \mathbb{R},\quad \text{as }\tau \to \infty .
\]
\end{theorem}

\begin{proof}
Integrating both sides of \eqref{C1}, we find
\begin{equation}
^C \mathfrak{D}_0^{\alpha }x(\tau) =b_2+\int_0^{\tau
}f(s,x(s))ds=b_2+\mathfrak{I}_0^{1}f(\tau ,x(\tau
)).  \label{cc}
\end{equation}
Applying $\mathfrak{I}_0^{\alpha }$ to both sides of \eqref{cc}, and using
Lemmas \ref{lem:2.4.1} and  \ref{pro:2.2.1}, we arrive at
\begin{equation}
x(\tau) =b_1+\frac{b_2}{\Gamma (\alpha +1)}\tau^{\alpha }+
\frac{1}{\Gamma (\alpha +1)}\int_0^{\tau }(\tau -s)^{\alpha
}f(s,x(s)) ds,\quad \tau \geq 0.  \label{C6.6}
\end{equation}
By using \eqref{C2} we obtain
\begin{equation}
| x(\tau) | \leq | b_1|
+\frac{| b_2| }{\Gamma (\alpha +1)}\tau^{\alpha }+
\frac{1}{\Gamma (\alpha +1) }\tau^{\alpha }\int_0^{\tau
}P(s)\varphi (| x(s)|) ds,\quad \tau \geq
0.  \label{C8}
\end{equation}
Applying Lemma \ref{lem:3.1.1} to \eqref{C8} we obtain
\[
| x(\tau) | \leq \begin{cases}
E^{-1}\big(E(| b_1| +\frac{|
b_2| }{\Gamma (\alpha +1)}) +\frac{1}{\Gamma (\alpha
+1) }\int_0^{\tau }P(s) ds\big) ,& 0\leq \tau <1
\\
\tau^{\alpha }E^{-1}\big(E(A) +\frac{1}{\Gamma (\alpha
+1) }\int_1^{\tau }s^{\alpha }P(s) ds\big) , &\tau \geq 1,
\end{cases}
\]
where
\begin{gather*}
A=| b_1| +\frac{| b_2| }{\Gamma
(\alpha +1)}+\frac{1}{\Gamma (\alpha +1) }\varphi (
E^{-1}(K)) \int_0^{1}P(s)ds, \\
K=E(| b_1| +\frac{| b_2| }{
\Gamma (\alpha +1)}) +\frac{1}{\Gamma (\alpha +1) }
\int_0^{1}P(s)ds<\infty .
\end{gather*}
From \eqref{C4} and the continuity of $P$ on $\mathbb{R}_{+}$, we see that
\begin{equation}
| x(\tau) | \leq \begin{cases}
C_1, & 0\leq \tau <1, \\
\tau^{\alpha }C_2, & \tau \geq 1,
\end{cases}  \label{C20}
\end{equation}
with
\begin{gather*}
C_1=E^{-1}\Big(E\big(| b_1| +\frac{|
b_2| }{\Gamma (\alpha +1)}\big) +\frac{1}{\Gamma (\alpha
+1) }\int_0^{1}P(s) ds\Big) <\infty , \\
C_2=E^{-1}\Big(E(A) +\frac{1}{\Gamma (\alpha +1)
}\int_1^{\infty }s^{\alpha }P(s) ds\Big) <\infty .
\end{gather*}
Next, it is clear that
\begin{equation}
\begin{aligned}
\int_0^{\tau }| f(s,x(s)) |ds
&\leq \int_0^{\tau }P(s)\varphi (| x(s)|) ds \\
&\leq \int_0^{1}P(s)\varphi (| x(s)|)ds+\int_1^{\tau }P(s)\varphi (| x(s)|) ds \\
&\leq \int_0^{1}P(s)\varphi (| x(s)|) ds
 +\int_1^{\tau }s^{\alpha }P(s)\varphi \big(\frac{|
x(s)| }{s^{\alpha }}\big) ds,\quad \tau >0.
\end{aligned} \label{C19}
\end{equation}
By \eqref{C20} and \eqref{C19}, we have
\begin{equation}
\lim_{\tau \to \infty }\int_0^{\tau }f(s,x(s)) ds<\infty .  \label{C21.5}
\end{equation}
On the other hand, integrating  \eqref{C1} yields
\begin{equation}
^C \mathfrak{D}_0^{\alpha }x(\tau) =b_2+\int_0^{\tau
}f(s,x(s)) ds,\quad \tau >0.  \label{C22}
\end{equation}
From \eqref{C21.5} and \eqref{C22}, we conclude
\[
\lim_{\tau \to \infty }{}^C \mathfrak{D}_0^{\alpha }x(\tau) =c,\quad c\in \mathbb{R}.
\]
Further, by Lemma \ref{lem:4.1.1}, we can write
\[
\lim_{\tau \to \infty }\frac{x(\tau) }{\tau^{\alpha }}
=\lim_{\tau \to \infty }\frac{^C \mathfrak{D}_0^{\alpha
}x(\tau) }{\Gamma (\alpha +1) }=a,
\]
for some real number $a$.
\end{proof}

\begin{example}\label{exm:4.1.1} \rm
All solutions of
\begin{equation}
(^C \mathfrak{D}_0^{\alpha }x) '(\tau
) =e^{-\tau }x^{r}(\tau) ,\quad 0<\alpha ,r\leq 1,\; \tau >0.  \label{C21}
\end{equation}
satisfy $\lim_{\tau \to \infty }\frac{x(\tau) }{\tau^{\alpha }}=a$, as
$\tau \to \infty $, for some real number $a$.
\end{example}

To prove this claim, let $\varphi (\tau) =\tau^{r}$ and $P(\tau) =e^{-\tau }$. Then
\[
\int_1^{\infty }s^{\alpha }P(s) ds\leq \int_0^{\infty
}s^{\alpha }e^{-s}ds=\Gamma (\alpha +1) <\infty .
\]
Obviously $\varphi $ is a nondecreasing and positive function with
\[
u\varphi (v) =uv^{r}\leq (vu)^{r}=\varphi (vu) ,\quad u\geq 1,\; v>0,
\]
and
\[
\int_0^{\infty }\frac{ds}{\varphi (s) }=\int_0^{\infty }\frac{ds}{s^{r}}=\infty .
\]
Then $\varphi \in \Phi$. All the conditions of Theorem \ref{thm:4.1.1}
are satisfied, therefore every solution $x$ of \eqref{C21} has satisfy
 $\lim_{\tau \to \infty }\frac{x(\tau) }{\tau^{\alpha }}=a$,
, $a\in \mathbb{R}$, as $\tau \to \infty $.

\section{Equations with fractional source terms\label{sec5}}

We study  problem \eqref{cd1} in the space $C_{1-\alpha }^{\alpha,1}[0,\infty)$
defined in \eqref{space of Caputo} with the following assumptions:
\begin{itemize}
\item[(C3)] $f(\tau ,v,w):[0,\infty)\times\mathbb{R}^{2}\to\mathbb{R}$
is so that $f(\cdot,v(\cdot) ,w(\cdot))\in C_{1-\alpha}[0,\infty)$ for every
$v,w\in AC[0,\infty)$.

\item[(C4)]
\begin{equation}
| f(\tau ,u(\tau) ,v(\tau))|
\leq F_1(\tau ,| u(\tau) |)
+F_2(\tau ,\tau^{\beta }| v(\tau) |) ,\quad \tau \geq 0,  \label{cd2}
\end{equation}
where $F_i\in M$, $i=1,2$.
\end{itemize}

\begin{lemma} \label{lem:4.2.3}
 Suppose that $f$ satisfies {\rm (C3), (4)}
and $x\in AC[0,\infty)$ is a solution of \eqref{cd1}. Then
\begin{equation}
\max \big\{ | x(\tau) | ,\tau^{\beta}|^C \mathfrak{D}_0^{\beta }x(\tau) |\big\}
\leq | b_1| +z(\tau) , \quad \tau >0,  \label{cd32}
\end{equation}
where
\begin{equation} \label{cd33}
\begin{gathered}
z(\tau) =C_2\tau^{\alpha }+C_3\tau^{\alpha
}\int_0^{\tau }[ F_1(s,| x(s)|)
+F_2(s,\tau^{\beta }|^C \mathfrak{D}_0^{\beta
}x(s) |) ] ds,\quad \tau >0,  \\
C_3=\max \big\{ \frac{1}{\Gamma (\alpha +1) },\frac{1}{\Gamma
(\alpha -\beta +1) }\big\} ,\quad
C_2=|b_2| C_3.
\end{gathered}
\end{equation}
\end{lemma}

\begin{proof}
Applying $\mathfrak{I}_0^{1}$ to \eqref{cd1}, we obtain
\begin{equation}
\begin{aligned}
^C \mathfrak{D}_0^{\alpha }x(\tau)
&= b_2+\mathfrak{I}_0^{1}f\big(\tau ,x(\tau) ,{}^C \mathfrak{D}
_0^{\beta }x(\tau)\big)    \\
&= b_2+\int_0^{\tau }f\big(s,x(s) ,{}^C \mathfrak{
D}_0^{\beta }x(s)\big) ds,\quad \tau >0.
\end{aligned}  \label{cd6}
\end{equation}
Next, we apply $\mathfrak{I}_0^{\alpha }$ to both sides of \eqref{cd6},
using Lemmas \ref{lem:2.2.2}, \ref{lem:2.4.1} and \ref{pro:2.2.1}, we
find
\begin{equation}
\begin{aligned}
x(\tau) &=b_1+\frac{b_2}{\Gamma (\alpha +1)}\tau^{\alpha }+
\mathfrak{I}_0^{1+\alpha }f\big(\tau ,x(\tau) ,{}^C \mathfrak{D}_0^{\beta }x(\tau)\big) \\
&=b_1+\frac{b_2}{\Gamma (\alpha +1)}\tau^{\alpha }+\frac{1}{\Gamma
(\alpha +1) }\int_0^{\tau }(\tau -s)^{\alpha
}f(s,x(s) ,{}^C \mathfrak{D}_0^{\beta }x(s)) ds,
\end{aligned}  \label{cd7}
\end{equation}
for $\tau >0$.
Thus, from \eqref{cd7} and \eqref{cd2} we have
\begin{equation}
\begin{aligned}
| x(\tau) |
& \leq | b_1| +\frac{| b_2| }{\Gamma (\alpha +1)}\tau^{\alpha }
 +\frac{\tau^{\alpha }}{\Gamma (\alpha +1) }\int_0^{\tau}\big| f\big(s,x(s) ,{}^C \mathfrak{D}
_0^{\beta }x(s)\big) \big| ds \\
&\leq | b_1| +C_2\tau^{\alpha }+C_3\tau^{\alpha
}\int_0^{\tau }\Big(F_1(s,| x(s)|)
+F_2\big(s,s^{\beta }|^C \mathfrak{D}_0^{\beta }x(
s) |\big)\Big) ds,  \label{cd9}
\end{aligned}
\end{equation}
for $\tau >0$.
By Lemma \ref{lem:4.2.1}, we see that
\begin{equation}
^C \mathfrak{D}_0^{\beta }x(\tau)
=\mathfrak{I}_0^{\alpha -\beta }(^C \mathfrak{D}_0^{\alpha }x(\tau)) .  \label{cd9.5}
\end{equation}
Let us insert the expression \eqref{cd6} into \eqref{cd9.5},
using Lemmas \ref{lem:2.2.2} and \ref{pro:2.2.1}, we have
\begin{align*}
^C \mathfrak{D}_0^{\beta }x(\tau)
&=\mathfrak{I}_0^{\alpha -\beta }\Big(b_2+\mathfrak{I}_0^{1}f\big(s,x(s) ,
{}^C \mathfrak{D}_0^{\beta }x(s)\big) \Big) (\tau) \\
&=\frac{b_2}{\Gamma (\alpha -\beta +1) }\tau^{\alpha -\beta }
 +\mathfrak{I}_0^{\alpha -\beta +1}f\big(\tau ,x(\tau) ,{}^C \mathfrak{D}_0^{\beta }x(\tau)\big)\\
&=\frac{b_2}{\Gamma (\alpha -\beta +1) }\tau^{\alpha -\beta } \\
&\quad +\frac{1}{\Gamma (\alpha -\beta +1) }\int_0^{\tau }(\tau-s)^{\alpha -\beta }f(s,x(s) ,{}^C
\mathfrak{D}_{0^{+}}^{\beta }x(s)) ds,\quad \tau >0.
\end{align*}
Then from this and  \eqref{cd2} we obtain the bound
\begin{equation}
\begin{aligned}
&\tau^{\beta }|^C \mathfrak{D}_0^{\beta }x(\tau)| \\
&\leq C_3| b_2| \tau^{\alpha }+C_3\tau
^{\alpha }\int_0^{\tau }| f(s,x(s) ,{}^C \mathfrak{D}_{0^{+}}^{\beta }x(s)) | ds\\
&\leq C_2\tau^{\alpha }+C_3\tau^{\alpha }\int_0^{\tau }\big(
F_1(s,| x(s)|) +F_2(s,s^{\beta}|^C \mathfrak{D}_0^{\beta }x(s) |)\big) ds,\quad \tau >0.
\end{aligned} \label{cd12}
\end{equation}
Relation \eqref{cd32} follows directly from \eqref{cd33}, \eqref{cd9}
and \eqref{cd12}.
\end{proof}

\begin{theorem} \label{thm:4.2.1}
 Suppose that $f$ satisfies (\textbf{C3)}-(\textbf{C4)} and
\begin{equation}
\int_0^{\infty }s^{\alpha }N_i(s) ds<\infty ,\quad
\int_0^{\infty }F_i(s,| b_1|)
ds<\infty \text{, }i=1,2.  \label{cd4}
\end{equation}
Then, every solution $x(\tau) $ of problem \eqref{cd1} has the
following property
\[
\lim_{\tau \to \infty }\frac{x(\tau) }{\tau^{\alpha }}
=a,\quad a\in \mathbb{R}.
\]
\end{theorem}

\begin{proof}
By using Lemma \ref{lem:4.2.3} we have
\begin{gather}
F_1(\tau ,| x(\tau)|) \leq F_1(
\tau ,| b_1| +z(\tau)) ,\quad \tau >0,  \label{cd34} \\
F_2\big(\tau ,\tau^{\beta }|^C \mathfrak{D}_0^{\beta }x(\tau) |\big)
\leq F_2(\tau ,z(\tau) +| b_1|) ,\quad \tau >0. \label{cd35}
\end{gather}
Taking into account \eqref{cd33}, \eqref{cd34} and \eqref{cd35} we arrive at
\[
z(\tau) \leq C_2\tau^{\alpha }+C_3\tau^{\alpha
}\int_0^{\tau }[ F_1(s,| b_1| +z(s)) +F_2(s,z(\tau) +|
b_1|) ] ds,\quad \tau >0.
\]
Then, by Lemma \ref{lem:4.2.2}, we find that
\begin{equation}
z(\tau) \leq C\tau^{\alpha },\quad \tau >0  \label{cd36}
\end{equation}
where
\begin{align*}
C&=\Big(C_2+C_3\int_0^{\infty }[ F_1(s,|b_1|) +F_2(s,| b_1|)] ds\Big)\\
&\quad \times \exp (C_3\int_0^{\infty}s^{\gamma }[ N_1(s) +N_2(s) ]ds) <\infty .
\end{align*}
It follows from Lemma \ref{lem:4.2.3} and \eqref{cd36} that
\begin{equation}
| x(\tau) | \leq | b_1|+C\tau^{\alpha },\quad
\tau^{\beta }|^C \mathfrak{D}_0^{\beta }x(\tau) |
\leq |b_1| +C\tau^{\alpha },\quad \tau >0.  \label{cd37}
\end{equation}
Again by  hypothesis \eqref{cd2} we have
\begin{align*}
| \int_0^{\tau }f(s,x(s) ,{}^C\mathfrak{D}_0^{\beta }x(s)) ds|
&\leq \int_0^{\tau }| f(s,x(s) ,{}^C \mathfrak{D}_0^{\beta }x(s)) | ds\\
&\leq \int_0^{\tau }[ F_1(s,| x(s)|) +F_2(s,s^{\beta }|^C \mathfrak{D}_0^{\beta }x(
s) |) ] ds,
\end{align*}
for $\tau >0$.
From this inequality and \eqref{cd37}, we obtain
\begin{align*}
&| \int_0^{\tau }f(s,x(s) ,{}^C \mathfrak{D}_0^{\beta }x(s)) ds| \\
&\leq \int_0^{\tau }[ F_1(s,| b_1| +Cs^{\alpha}) +F_2(s,| b_1| +Cs^{\alpha })] ds \\
&=\int_0^{\tau }\Big\{ F_1(s,| b_1| +Cs^{\alpha }) -F_1(s,| b_1|)+F_1(s,| b_1|) \\
&\quad  +F_2(s,|b_1| +Cs^{\alpha }) -F_2(s,|b_1|) +F_2(s,| b_1|)\Big\} ds,\quad \tau >0.
\end{align*}
As the functions $F_i$, $i=1,2$ are in $M$, we obtain
\begin{equation}
\begin{aligned}
&\big| \int_0^{\tau }f(s,x(s) ,{}^C
\mathfrak{D}_0^{\beta }x(s)) ds\big| \\
&\leq C\int_0^{\tau }s^{\alpha }[ N_1(s) +N_{s}(
s) ] ds+\int_0^{\tau }[ F_1(s,|
b_1|) +F_2(s,| b_1|)] ds<\infty ,
\end{aligned}  \label{cd23}
\end{equation}
where we have used \eqref{cd4}. Then
\[
\lim_{\tau \to \infty }\int_0^{\tau }f(s,x(s) ,{}^C \mathfrak{D}_0^{\beta }x(s)) ds<\infty .
\]
By \eqref{cd6} we conclude that there is $b\in \mathbb{R}$ such that
$\lim_{\tau \to \infty }{}^C \mathfrak{D}_0^{\alpha}x(\tau) =b$.
Further, by Lemma \ref{lem:4.1.1}, we deduce that
\[
\lim_{\tau \to \infty }\frac{x(\tau) }{\tau^{\alpha }}
=\lim_{\tau \to \infty }\frac{^C \mathfrak{D}_0^{\alpha
}x(\tau) }{\Gamma (\alpha +1) }=a,
\]
and the proof is now complete.
\end{proof}

\section{Boundedness} \label{sec6}

We consider the fractional differential problem \eqref{sc1} in the space
\begin{equation}
C^{\alpha }[0,\infty)=\left\{ x\in AC[0,\infty):{}^C \mathfrak{D}
_0^{\alpha }x\in C[0,\infty)\right\} .  \label{space decay C}
\end{equation}
We assume the following conditions:
\begin{itemize}
\item[(C5)] $f:[0,\infty)\times\mathbb{R}^{2}\to\mathbb{R}$ is so that
$f(\cdot,v(\cdot) ,w(\cdot)) \in C[0,\infty)$ for every $v$, $w$ in $C[0,\infty)$.

\item[(C6)]
\begin{equation}
| f(\tau ,u,v) | \leq \tau^{\gamma }h(
\tau) \varphi _1(| u(\tau) |) \varphi _2(| v(\tau) |) ,\quad \tau >0,  \label{sc2}
\end{equation}
where $h$, $\varphi _1$, $\varphi _2:\mathbb{R}_{+}\to \mathbb{R}_{+}$
are continuous functions with $\varphi _i$, $i=1,2$, are
nondecreasing functions and $h\in L_{q}(0,\infty) $ for some
$q>\frac{1}{\alpha -\beta }$, $\gamma =\frac{1}{q}-\alpha $.
\end{itemize}

\begin{lemma} \label{lem:esc}: Suppose that $f$ satisfies {(C5), (C6)}
and $x\in AC[0,\infty)$ is a solution of \eqref{sc1}.
Then
\begin{equation}
\max \big\{ | x(\tau) | ,|^C
\mathfrak{D}_0^{\beta }x(\tau) | \big\} \leq
z(\tau) ,\quad \tau \geq \tau _0>0,  \label{4.35}
\end{equation}
where
\begin{equation}
z(\tau) =| b| +K_1\Big(\int_0^{\tau
}h^{q}(s) \varphi _1^{q}(| x(s)
|) \varphi _2^{q}(|^C \mathfrak{D}
_0^{\beta }x(s) |) ds\Big)^{1/q},
\quad \tau >0,  \label{4.36}
\end{equation}
and
\begin{gather*}
K_1=\max \big\{\frac{K_{1+p(\alpha -1) ,p\gamma }^{1/p}}{
\Gamma (\alpha) },\frac{K_{1+p(\alpha -\beta -1)
,p\gamma }^{1/p}}{\Gamma (\alpha -\beta) \tau _0^{\beta }}\big\} ,\\
 K_{\alpha ,\beta }=\frac{\Gamma (\beta +1)
\Gamma (\alpha) }{\Gamma (\alpha +\beta +1) },\quad
\frac{1}{p}+\frac{1}{q}=1.
\end{gather*}
\end{lemma}

\begin{proof}
Applying $\mathfrak{I}_0^{\alpha }$ to \eqref{sc1} and taking into account
Lemma \ref{lem:2.4.1}, we have
\begin{equation}
x(\tau) =b+\frac{1}{\Gamma (\alpha) }
\int_0^{\tau }(\tau -s)^{\alpha -1}f(s,x(s)
,{}^C \mathfrak{D}_0^{\beta }x(s)) ds,\quad \tau
>0.  \label{sc4}
\end{equation}
Using the inequality \eqref{sc2}, we obtain
\begin{equation}
| x(\tau) | \leq | b| +
\frac{1}{\Gamma (\alpha) }\int_0^{\tau }(\tau -s)
^{\alpha -1}s^{\gamma }h(s) \varphi _1(|
x(s) |) \varphi _2(|^C
\mathfrak{D}_0^{\beta }x(s) |) ds,  \label{sc5}
\end{equation}
for $\tau >0$.
It follows from the assumptions $\beta <\alpha$, $q>\frac{1}{\alpha -\beta }$
and $\gamma =\frac{1}{q}-\alpha $ that
$p(\alpha -1) +1\geq p(\alpha -\beta -1) +1>0$ and
$p\gamma +1=p(\frac{1}{q}-\alpha) +1=p(1-\alpha) >0$.
Then, we apply Lemma \ref{lem:rec2}, to obtain
\begin{equation}
\begin{aligned}
| x(\tau) |
&\leq | b| + \frac{1}{\Gamma (\alpha) }K_{p(\alpha -1)
+1,p\gamma }^{1/p}\tau^{\alpha +\gamma -1/q}
\Big(\int_0^{\tau }h^{q}(s) \varphi _1^{q}(| x(s)
|) \varphi _2^{q}(|^C \mathfrak{D}
_0^{\beta }x(s) |) ds\Big)^{1/q} \\
&\leq | b| +K_1\Big(\int_0^{\tau }h^{q}(s) \varphi _1^{q}(| x(s) |
) \varphi _2^{q}(|^C \mathfrak{D}_0^{\beta}x(s) |) ds\Big)^{1/q},\quad \tau
>0.
\end{aligned} \label{4.39}
\end{equation}
Also, by Lemma \ref{lem:4.2.1}, we conclude that
\begin{equation}
\begin{aligned}
^C \mathfrak{D}_0^{\beta }x(\tau)
& =\mathfrak{I}_0^{\alpha -\beta }~^C \mathfrak{D}^{\alpha }x(\tau) =
\frac{1}{\Gamma (\alpha -\beta) }\int_0^{\tau }(\tau
-s)^{\alpha -\beta -1} ~^C \mathfrak{D}_0^{\alpha }x(s) ds \\
&=\frac{1}{\Gamma (\alpha -\beta) }\int_0^{\tau }(\tau
-s)^{\alpha -\beta -1}f(s,x(s) ,{}^C \mathfrak{D}
_0^{\beta }x(s)) ds,\quad \tau >0.
\end{aligned}  \label{sc9}
\end{equation}
In view of \eqref{sc2}, we have
\[
|^C \mathfrak{D}_0^{\beta }x(\tau) |
\leq \frac{1}{\Gamma (\alpha -\beta) }\int_0^{\tau }(
\tau -s)^{\alpha -\beta -1}s^{\gamma }h(s) \varphi
_1(| x(s) |) \varphi
_2(|^C \mathfrak{D}_0^{\beta }x(s)
|) ds,
\]
for $\tau >0$.
Again, from Lemma \ref{lem:rec2}, we find
\begin{equation}
\begin{aligned}
&|^C \mathfrak{D}_0^{\beta }x(\tau) | \\
&\leq \frac{K_{p(\alpha -\beta -1) +1,p\gamma }^{1/p}}{\Gamma
(\alpha -\beta) }\tau^{\alpha -\beta +\gamma -1/q}
\Big(\int_0^{\tau }h^{q}(s) \varphi _1^{q}(|x(s) |) \varphi _2^{q}(|^C
\mathfrak{D}_0^{\beta }x(s) |) ds\Big)^{1/q} \\
&\leq \frac{K_{p(\alpha -\beta-1) +1,p\gamma }^{1/p}}{\Gamma (\alpha -\beta) }\tau
^{-\beta }\big(\int_0^{\tau }h^{q}(s) \varphi _1^{q}(
| x(s) |) \varphi _2^{q}(|^C \mathfrak{D}_0^{\beta }x(s) |) ds\Big)^{1/q} \\
&\leq K_1(\int_0^{\tau}h^{q}(s) \varphi _1^{q}(| x(s)
|) \varphi _2^{q}(|^C \mathfrak{D}_0^{\beta }x(s) |) ds)^{1/q},
\quad \tau \geq \tau _0>0.
\end{aligned}  \label{4.42}
\end{equation}
Therefore \eqref{4.35} follows from  \eqref{4.36}, \eqref{4.39} and \eqref{4.42}.
\end{proof}


\begin{theorem}\label{thm:sc}
 Assume that $f$ satisfies {\rm (C5), (C6)}. Then, any
solution $x$ of \eqref{sc1} satisfies
\[
| x(\tau) | \leq C, \quad
|^C \mathfrak{D}_0^{\beta }x(\tau) | <C,
\]
for some positive constant $C,\tau >0$,
provided that
\[
\int_{\mathfrak{\xi }_0}^{\infty }\frac{ds}{\varphi _1^{q}(
s^{1/q}) \varphi _2^{q}(s^{1/q}) }=\infty ,\quad
\mathfrak{\xi }_0>0.
\]
\end{theorem}

\begin{proof}
In view of Lemma \ref{lem:esc} we have
\begin{equation}
\varphi _1(| x(\tau) |) \leq
\varphi _1(z(\tau)), \quad
_2(|^C \mathfrak{D}_0^{\beta }x(\tau)
|) \leq \varphi _2(z(\tau)) ,
\quad \tau >0.  \label{sc13}
\end{equation}
From this inequality and \eqref{4.36}, we obtain
\begin{equation}
z(\tau) \leq | b| +K_1\Big(
\int_0^{\tau }h^{q}(s) \varphi _1^{q}(z(s)
) \varphi _2^{q}(z(s)) ds\Big)^{1/q},\quad \tau >0.  \label{sc14}
\end{equation}
Therefore, Lemma \ref{lem:res1} implies
\[
z(\tau) \leq \Big[ E^{-1}\Big(E(2^{q-1}K_1)+2^{q-1}K_2\int_0^{\tau }h^{q}(s) ds\Big) \Big]
^{1/q}<\infty ,
\]
because $h\in L_{q}(0,\infty) $. This completes the proof.
\end{proof}

\begin{example} \label{exsc3} \rm
 Consider the problem
\begin{equation}
\begin{gathered}
^C \mathfrak{D}_0^{2/3}x(\tau)
 =\tau^{1/q-2/3}e^{-\lambda \tau }(x(\tau))^{3/5}
 \big(^C \mathfrak{D}_0^{1/3}x(\tau)\big)^{1/3}
 \big(\cos (^C\mathfrak{D}_0^{1/3}x)\big) \quad \tau >0, \\
x(0) =b,\quad q>3,\quad \lambda >0.
\end{gathered}  \label{exsc}
\end{equation}
\end{example}

Let $\varphi _1(\tau) =\tau^{3/5}$,
$\varphi _2(\tau) =\tau^{1/3}$ and $h(\tau) =e^{-\lambda \tau }$,
$\gamma =1/q-2/3$. Then $h\in L_{q}(0,\infty) $ and
\[
\int_{\mathfrak{\xi }_0}^{\infty }\frac{ds}{\varphi _1^{q}(s^{
\frac{1}{q}}) \varphi _2^{q}(s^{1/q}) }
=\int_{\mathfrak{\xi }_0}^{\infty }\frac{ds}{s^{3/5}s^{1/3}}=\int_{\mathfrak{\xi }
_0}^{\infty }\frac{ds}{s^{14/15}}=\infty .
\]
Then, by Theorem \ref{thm:sc}, we deduce that any solution $x$ of \eqref{exsc} satisfies
\[
| x(\tau) | \leq C, \quad |^C \mathfrak{D}_0^{\beta }x(\tau) | <C,
\]
for $\alpha =2/3$, $\beta =1/3$, and $\tau >0$.

\subsection*{Acknowledgements}
M. D. Kassim wants to thank Imam
Abdulrahman Bin Faisal University for its  support and facilities.
N. E. Tatar is grateful for the financial support and the
facilities provided by King Fahd University of Petroleum and Minerals
through project number IN181008.

\end{document}